\documentclass[12pt]{article}
\usepackage{amsfonts,amsmath,color,amsthm,amssymb}
\usepackage{tikz}
\usepackage{pstricks}
\usepackage{graphicx}
\usepackage{epstopdf}
\usepackage{permute}
\usepackage{latexsym}
\usepackage{mathrsfs}
\usepackage{hyperref}
\usepackage[utf8]{inputenc}
\usepackage{lscape}
\usepackage{longtable}
\usepackage{algorithm2e} 
\usepackage{enumitem} 
\usepackage{array}
\usepackage{tikz}
\usepackage{tikz-qtree,tikz-qtree-compat}
\usetikzlibrary{positioning,decorations.pathreplacing}

\title{The Expected Number of Distinct Consecutive Patterns in a Random Permutation}

\author{Austin Allen\\
Carnegie Mellon University\and Dylan Cruz Fonseca\\
University of Puerto Rico, R\'io Piedras\and
Veronica Dobbs\\ Marquette University\and 
Egypt Downs\\
Virginia State University\and
Evelyn Fokuoh and Anant Godbole\\ East Tennessee State University\and Sebasti\'an Papanikolaou Costa\\ Universidad Ana G. M\'endez\and Christopher Soto\\ Queens College, CUNY\and Lino Yoshikawa\\ University of Hawaii at Hilo}

\newtheorem{thm}{Theorem}[section]
\newtheorem{lem}[thm]{Lemma}

\newtheorem{dfn}{Definition}[section]
\def\p{\mathbb P}
\def\e{\mathbb E}
\def\lr{\left(}
\def\rr{\right)}
\def\lc{\left\{}
\def\rc{\right\}}
\newcommand{\be}{\begin{equation}}
\newcommand{\ee}{\end{equation}}
\newcommand{\beq}{\begin{eqnarray}}
\newcommand{\eeq}{\end{eqnarray}}
\date{}
\begin{document}
\maketitle
\begin{abstract}
Let $\pi_n$ be a uniformly chosen random permutation on $[n]$.  Using an analysis of the probability that two overlapping consecutive $k$-permutations are order isomorphic, we show that the expected number of distinct consecutive patterns in $\pi_n$ is $\frac{n^2}{2}(1-o(1))$.  This exhibits the fact that random permutations pack consecutive patterns near-perfectly.
\end{abstract}
\section{Introduction}

Let $\pi_n$ be a permutation on $[n]$.  We say that $\pi_n$ {\it contains} a permutation $\mu_k$ of length $k$ if there are $k$ indices $n_1<n_2<\ldots <n_k$ such that $(\pi_{n_1}, \pi_{n_2},\ldots,\pi_{n_k})$ are in the same relative order as $(\mu_1,\mu_2,\ldots, \mu_k)$. We say that $\pi_n$ {\it consecutively contains} the permutation $\mu_k$  if there are $k$ consecutive indices $(m,m+1,\ldots,m+k-1)$ such that $(\pi_{m}, \pi_{m+1},\ldots,\pi_{m+k-1})$ are in the same relative order as $(\mu_1,\mu_2,\ldots, \mu_k)$. Let $\phi(\pi_n)$ be the number of distinct consecutive patterns of all lengths $k; 1\le k\le n$, contained in $\pi_n$.  We focus on the case where $\pi_n$ is a {\it uniformly chosen random permutation} on $[n]$, denote the random value of $\phi(\pi_n)$ by $X$, and study, in this paper, its expected value $\e(X)$.  
\subsection{Distinct Subsequences and Non-Consecutive Patterns}
First we summarize results on the extremal values of $\psi(\pi_n)$, where $\psi(\pi_n)$ is the number of distinct ({\it and not necessarily consecutive}) patterns contained in $\pi_n$.  The identity permutation reveals that
\[\min_{\pi_n\in S_n}\psi(\pi_n)=\min_{\pi_n\in S_n}\phi(\pi_n)=n+1,\]
since the embedded patterns are $\emptyset, 1, 12,\ldots,(12\ldots n)$.  On the other hand, motivated by a question posed by Herb Wilf at the inaugural Permutation Patterns meeting, held in Dunedin in 2003 (PP2003), several authors have studied the maximum value of $\psi(\pi_n)$.  First we have the trivial pigeonhole bound
\begin{equation}\max_{\pi_n\in S_n}\psi(\pi_n)\le\sum_{k=1}^n\min\lr{n\choose k}, k!\rr\sim 2^n,\end{equation}
which was mirrored soon after PP2003 by Coleman \cite{c}:
\begin{equation}\max_{\pi_n\in S_n}\psi(\pi_n)\ge 2^{n-2{\sqrt n}+1}\enspace(n=2^k);\end{equation}
which led to 
\[\lr\max_{\pi_n\in S_n}\psi(\pi_n)\rr^{1/n}\to2.\]  A team of researchers began to see if this (surprising) bound could be improved.  This led to the result in \cite{a} that
\begin{equation}\max_{\pi_n\in S_n}\psi(\pi_n)\ge2^n\lr1-6{\sqrt{n}}2^{-{\sqrt n}/2}\rr,\end{equation}
and thus to the conclusion that $\max_{\pi_n\in S_n}\psi(\pi_n)\sim 2^n$.  Alison Miller improved both the upper and lower bounds (2) and (3), showing in \cite{m} that
\begin{equation}2^n-O(n^22^{n-{\sqrt{2n}}})\le\max_{\pi_n\in S_n}\psi(\pi_n)\le 2^n-\Theta(n2^{n-{\sqrt{2n}}}).\end{equation} 
By extracting the constants in (4) and conducting an asymptotic analysis, Fokuoh showed in \cite{f} that the trivial upper bound actually performs better that the one in (4) for small and not-too-small values of $n$, though, of course (4) does better asymptotically.  

In \cite{ba} the authors studied the expected number $\e(\xi(W))$ of distinct subsequences contained in the word $W=W_n$ obtained when $n$ letters $s_1,\ldots,s_n$ are independently generated from a $d$-letter alphabet -- with the $i$th letter being ``typed" with probability $\alpha_i$ (they also covered the two-state Markov case).  In the simplest case, when $d=2$, it was shown in \cite{ba} (with $\alpha_1:=\alpha$) that
asymptotically 
$$\e(\xi(W)) \sim k \big(1+\sqrt{\alpha(1-\alpha)}\big)^n,$$
which contains the earlier result from \cite{fl} that in the equiprobable case, $\e(\xi(W))\sim k(\frac{3}{2})^n$ for a constant $k$.

The fact that $\e(\xi(W))\sim A^n$ for $A<2$ might suggest that the same is true for $\e(\psi(\pi_n))$.  But consider the following argument.  Since
\[k!\gg {n\choose k}\] for large $k$, it would seem reasonable, via a heuristic ``balls in boxes" argument that most or all of the patterns of large size contained in $\pi_n$ would be distinct.  It was accordingly conjectured in \cite{f} that
\begin{equation}\e(\psi(\pi_n))\sim 2^n.\end{equation}

While we are unable to prove that (5) holds, we show in this paper that the following is true for the number $X$ of distinct consecutive permutations contained in a random permutation:

\bigskip

\centerline{\bf Main Theorem}
\[\e(X)=\max_{\pi_n\in S_n}(\phi(\pi_n))(1-o(1))=\frac{n^2}{2}(1-o(1)).\]

\section{Proof of Main Theorem}  
\begin{lem}For any $\pi_n\in S_n$, 
\[\phi(\pi_n)\le \sum_{k=1}^n\min\{(n-k+1), k!\}\le \frac{n^2}{2}(1+o(1)).\]
\end{lem}
\begin{proof}  There are $k!$ permutation patterns of length $k$.  However, not all of these can be present unless the number of consecutive positions of length $k$, namely $(n-k+1)$, provide ``enough room" for this to occur, i.e., if $(n-k+1)\ge k!$  This proves the first inequality.
Next note that
\begin{equation}
\phi(\pi_n)\le\sum_{k=1}^n\min\{(n-k+1), k!\}\le\sum_{k=1}^n (n-k+1)=\sum_{k=1}^nk\le\frac{n^2}{2}(1+o(1)).
\end{equation}
This completes the proof.\end{proof}
\begin{lem}
\[\sum_{k=1}^n\min\{(n-k+1), k!\}\ge \frac{n^2}{2}(1-o(1)).\]
\end{lem}
\begin{proof}
The equation $(n-k+1)=k!$, by Stirling's approximation, holds if
\[{\sqrt{2\pi k}}\lr\frac{k}{e}\rr^k(1+o(1))+k=n+1, \]
which is true if and only if
\[\lr\exp\{k\log k-k+(1/2)(\log k+\log(2\pi))+o(1)\}\rr(1+o(1))=\exp\{\log n(1+o(1))\}.\]
A good approximation to the above is the solution to $k\log k=\log n$, namely
\[k=\frac{\log n}{\log\log n}(1+o(1)):=a_n,\]
($a_n$ will be critical in what follows), which yields
\[\sum_{k=1}^n\min\{(n-k+1), k!\}\ge \sum_{k=a_n}^n(n-k+1)\sim\frac{(n-a_n)^2}{2}=\frac{n^2}{2}(1-o(1)),\]
as asserted.
\end{proof}
We mention that the evidence in support of the Main Theorem  is strong, as evidenced by the following data for small $n$ (the evidence in support of (5) is not as strong; see \cite{f})
\begin{center}
\begin{tabular}{ |c|c|c|c| }%block 2 start
 \hline
 $n$ & $\sum_{k=1}^{n}$min$(n-k+1,k!)$ & Bound attained (Y/N)& $\e(X)$\\ 
 \hline
 3&4&Yes&3.67\\
 \hline
 4&6&Yes&5.83\\
 \hline
 5&9&Yes&8.7\\
 \hline
 6&13&Yes&12.33\\
 \hline
 7&18&Yes&16.78\\
 \hline
 8&24&Yes&22.08\\
 \hline
\end{tabular}
\end{center}
For example, the permutation 14325 contains the 9=$\sum_{k=1}^5\min\{k!, 6-k\}$ patterns 1, 12, 21, 132, 321, 213, 1432, 3214, and 14325. 

A study of patterns that occur in consecutive positions in a permutation is not new.  For example, the so called vincular patterns partially follow this scheme.  Far more relevant to this paper, however, is the work of \cite{b} and \cite {d} on non-consecutive permutations that touches on some of the aspects of this paper.
\subsection{Auxiliary Random Variables}
We will attack our problem using several auxiliary random variables to analyze the behavior of $X$.  First we define 
$X_k$ to be the number of distinct consecutive permutations of length $k$ contained in a random $n$-permutation.  We note that for some $k_0\ge a_n$ to be chosen later,
\begin{equation}
\e(X)=\sum_{k=1}^n\e(X_k)\ge \sum_{k=k_0}^n\e(X_k)=\sum_{k=k_0}^n\lr(n-k+1)-\e(Y_k)\rr,
\end{equation}
where $Y_k$ denotes the number of $k$ permutations that are consecutively contained in the random $n$-permutation as ``repeats".  

The variables $Y_k$ are difficult to work with directly, so we turn to yet another variable
\begin{equation}Z_k=\sum_{j=1}^{n-k+1}\sum_{l=0}^{k-1}I_{j,l,k},\end{equation}
where for $l\ge 1$, the indication variable $I_{j,l,k}$ equals one iff $(\pi_j,\ldots,\pi_{j+k-1})$ and  
$(\pi_{j+k-l}\ldots,\pi_{j+2k-l-1})$ are order isomorphic ($I_{j,l,k}$ equals zero otherwise.)  For $l=0$, $I_{j,l,k}=1$ if two non-overlapping sets of consecutive indices have the same pattern.  See Figure 1.  In other words $I_{j,l,k}$ equals one iff two consecutive sets of $k$ positions along the random $n$-permutation, overlapping in $l$ positions, are order isomorphic.  For two permutations $\eta_1$ and $\eta_2$ that are order isomorphic, we write $\eta_1\simeq \eta_2$.  Next, we set
\begin{equation}
Z=\sum_{k=k_0}^nZ_k=\sum_{k=k_0}^n\sum_{j=1}^{n-k+1}\sum_{l=0}^{k-1}I_{j,l,k},
\end{equation}
and make the crucial observation that 
\begin{equation}
\{Y\ge 1\}\subseteq\{Z\ge 1\},
\end{equation}
where 
\begin{equation}
Y=\sum_{k=k_0}^nY_k,
\end{equation}
so that 
\begin{equation}\p(Y\ge1)\le\p(Z\ge 1)\le \e(Z).
\end{equation}
Next notice that (by (12))
\begin{eqnarray}
\e(Y)&=&\sum_{j=1}^{n^2/2}j\p(Y=j)\nonumber\\&\le&\frac{n^2}{2}\sum_{j\ge 1}\p(Y=j)\nonumber\\
&=&\frac{n^2}{2}\p(Y\ge 1)\nonumber\\&\le&\frac{n^2}{2}\e(Z).
\end{eqnarray}
\begin{dfn}
A $k$-permutation is said to be $l$-good if it is possible for $I_{j,l,k}$ to equal one, i.e., if $\p(I_{j,l,k}=1)>0$.
\end{dfn}
The proof of the next lemma is obvious.
\begin{lem} For $l=0$ any $k$-permutation is 0-good, and $\p(I_{j,0,k}=1)=\frac{1}{k!}$.
\end{lem}
\begin{lem} For $l=k-1$, the only good permutations are the monotone ones.
\end{lem}
\begin{proof}If, without loss of generality, $\mu=(12\ldots k)$, then it is clear that $\p(I_{j,k-1,k}=1)=\frac{1}{(k+1)!}$.  On the other hand if $\mu\ne(12\ldots k)$; $\mu\ne(k(k-1)\ldots1)$, then without loss of generality, there exist indices $i, i+1, i+2$ such that $\pi(i+1)>\max\{\pi(i), \pi(i+2)\}$.  But then, for the permutation shifted to the right by 1, $\pi(i)>\pi(i+1)$, so that we cannot have $I_{j,k-1,k}=1$.
\end{proof}
\begin{lem}
For $1\le l\le \frac{k}{2}$, \begin{equation}\p(I_{j,l,k}=1)\le\frac{3^k}{k!}.\end{equation}
\end{lem}
\begin{proof}
Denote the two permutations in question, overlapping in $l$ positions, by $\eta_1$ and $\eta_2$.   We must have the following situation shown in Figure 1:
\begin{figure}
    \centering
    
	    {\def\XUNIT{.75} \def\YUNIT{.75}
	 \begin{tikzpicture}[scale=\YUNIT]
    	\node at (1,1)(1) {X};
    	\node at (2,1)(2) {X};
    	\node at (3,1)(3) {X};
    	\node at (4,1)(4) {X};
    	\node at (5,1)(5) {X};
    	\node at (6,1)(6) {X};
    	\node at (7,1)(7) {X};
    	\node at (8,1)(8) {X};
    	\node at (9,1)(9) {X};
    	\node at (10,1)(10) {X};
    	\node at (8,0)(11) {O};
    	\node at (9,0)(12) {O};
        \node at (10,0)(13) {O};
        \node at (11,0)(14) {O};
        \node at (12,0)(15) {O};
        \node at (13,0)(16) {O};
        \node at (14,0)(17) {O};
        \node at (15,0)(18) {O};
        \node at (16,0)(19) {O};
        \node at (17,0)(20) {O};
       
        \draw[decorate, decoration={brace,raise=2pt,amplitude=10pt, mirror}](1.south west)--(3.south east);
        \draw[decorate, decoration={brace,raise=2pt,amplitude=10pt, mirror}](4.south west)--(7.south east);
        \draw[decorate, decoration={brace,raise=2pt,amplitude=10pt, mirror}](11.south west)--(13.south east);
        \draw[decorate, decoration={brace,raise=2pt,amplitude=10pt, mirror}](14.south west)--(17.south east);
        \draw[decorate, decoration={brace,raise=2pt,amplitude=10pt, mirror}](18.south west)--(20.south east);
        
        \node at (2,2) {$\leftarrow l \rightarrow $};
        \node at (5.5,2) {$\leftarrow k-2l \rightarrow $};
        
        \node at (2,-0.5) {$\rho_1$};
        \node at (5.5,-0.5) {$\omega_1$};
        \node at (9,-1.5) {$\rho_2$};
        \node at (12.5,-1.5) {$\omega_2$};
        \node at (16,-1.5) {$\rho_3$};
	\end{tikzpicture}    
	}
    \caption{Consecutive Overlapping Permutations, $l\leq \frac{k}{2}$}
    \label{fig:my_label}
\end{figure}
\[\rho_1\simeq\rho_2\simeq\rho_3,\]
and
\[w_1\simeq w_2\]
in order for $\eta_1$ and $\eta_2$ to be {\it consistent with being order isomorphic}.  Thus
\begin{eqnarray}\p(\eta_1\simeq\eta_2)&\le&\p(\eta_1,\eta_2\ {\rm are\ consistent})\nonumber\\
&\le&\p(\rho_1\simeq\rho_2\simeq\rho_3, w_1\simeq w_2)\nonumber\\
&\le&\frac{1}{l!^2}\frac{1}{(k-2l)!}\nonumber\\
&\le&\frac{3^k}{k!},\end{eqnarray}
where the last line of (15) follows from the fact that
\[\frac{k!}{(l!^2)(k-2l)!}\le\frac{k!}{(k/3)!^3}\sim K\cdot \frac{3^k}{k}\]
for some constant $K$, by Stirling's approximation.  Note that the bound in (15) is uniform, i.e., independent of the value of $l\le k/2$, but perhaps more importantly, it is obtained without fussing about which values of integers might actually occupy the $2k-l$ positions of $\eta_1\cup\eta_2$.  An analysis that analyzes these values gets very complicated very rapidly and so we will settle for the upper bound in (15).  The same kind of analysis, in which we use consistency with order isomorphism as a driving method, is used for $l>\frac{k}{2}$, which we turn to next.
\end{proof}
\begin{lem} For $k-2\ge l>\frac{k}{2}$, 
\begin{equation}\p(I_{j,l,k}=1)\le\lr\frac{1}{(k-l)!}\rr^{\frac{k}{k-l}-1}.\end{equation}
\end{lem}
\begin{proof}  We first illustrate the idea of the proof for $k=8; l=6$; see Figure 2.  If the first two elements of the permutation $\eta_1$ form the pattern $AB$, then so must the first two elements of $\eta_2$, which are also the third and fourth elements of $\eta_1$ -- forcing the third and fourth elements of $\eta_2$ to form an $AB$ pattern too.  This repetition of the $AB$ pattern persists till we reach the end of $\eta_2$, for a total of four induced $AB$ patterns caused by the first two elements of $\eta_1$.  Thus 
\[\p(I_{j,k,l}=1)\le\lr\frac{1}{2!}\rr^4.\]
\begin{figure}
    \centering
    
	    {\def\XUNIT{.75} \def\YUNIT{.75}
	 \begin{tikzpicture}[scale=\YUNIT]

	\node at (0,1)(0) {$\eta_1$:};
    	\node at (1,1)(1) {A};
    	\node at (2,1)(2) {B};
    	\node at (3,1)(3) {A};
    	\node at (4,1)(4) {B};
    	\node at (5,1)(5) {A};
    	\node at (6,1)(6) {B};
    	\node at (7,1)(7) {A};
    	\node at (8,1)(8) {B};
	\node at (2,0)(9) {$\eta_1$:};
    	\node at (3,0)(10) {A};
    	\node at (4,0)(11) {B};
        \node at (5,0)(12) {A};
        \node at (6,0)(13) {B};
        \node at (7,0)(14) {A};
        \node at (8,0)(15) {B};
        \node at (9,0)(16) {A};
        \node at (10,0)(17) {B};

	\end{tikzpicture}    
	}
    \caption{Consecutive Overlapping Permutations, $l\geq \frac{k}{2}$}
    \label{fig:my_label}
\end{figure}
In general the pattern in the first $k-l$ positions of $\eta_1$ is repeated $\lfloor\frac{k}{k-l}\rfloor\ge{\frac{k}{k-l}-1}$ times, each of which has a probability $\frac{1}{(k-l)!}$.  This completes the proof.
\end{proof}
\subsection{Putting it all Together}  For $k\ge k_0$ ($k_0 $ is still to be specified), we seek to find $\sum_{l=0}^{k-1}\p(I_{j,l,k}=1)$.  We address the case of $l=0$,  $l=k-1$ first.  By Lemmas 2.3 and 2.4, we have
\begin{equation}
\p(I_{j,0,k}=1)\le\frac{1}{k!},
\end{equation}
and 
\be
\p(I_{j,k-1,k}=1)\le\frac{2}{(k+1)!}.
\ee
Now, since there are $\le n$ consecutive positions disjoint from the $j$th set, we see that
\be
\sum_{k=k_0}^n\sum_j\p(I_{j,0,k}=1)\le
\frac{n^3}{k!},
\ee
and
\be
\sum_{k=k_0}^n\sum_j\p(I_{j,k-1,k}=1)\le
\frac{2n^2}{(k+1)!}.
\ee
For the case of small overlaps, Lemma 2.5 gives
\be
\sum_{k=k_0}^n\sum_j\sum_{l=1}^{k/2}\p(I_{j,k,l}=1)\le 
\frac{n^2k3^k}{k!}.
\ee
Finally, for the large overlap case, we have
\beq
&&\sum_{k=k_0}^n\sum_j\sum_{l=(k/2)}^{k-2}\p(I_{j,k,l}=1)\nonumber\\&\le& n^2\sum_l \lr\frac{1}{(k-l)!}\rr^{\frac{k}{k-l}-1}=n^2\sum_l \lr\frac{1}{(k-l)!}\rr^{\frac{l}{k-l}}\nonumber\\
&=&n^2\lc\lr\frac{1}{2!}\rr^{(k-2)/2}+\lr\frac{1}{3!}\rr^{(k-3)/3}+\ldots+\lr\frac{1}{(k/2)!}\rr
\rc.
\eeq
In (22), the $l=k-2$ term is $2(1/2)^{k/2}$, which we treat separately.  For the other terms we use the inequality $r!\ge{\sqrt{2\pi r}}(r/e)^r$ to provide the estimate
\be\lr\frac{1}{(k-l)!}\rr^{\frac{l}{k-l}}\le\lc\lr\frac{1}{(k-l)!}\rr^{\frac{1}{k-l}}\rc^{k/2}\le\lr\frac{e(1+o(1))}{k-l}\rr^{k/2}\le(0.96)^k.\ee
Notice that we could have been much more precise in dealing with the large overlap case.  Recognizing that (23) yields the largest upper bound of the terms in equations (19) to (23), we conclude from (9) that
\be
\e(Z)\le 6n^3(0.96)^k,
\ee
which via (13) gives
\be
\e(Y)\le 3n^5(0.96)^k,
\ee
so that by (7) we have that
\be
\e(X)\ge \lr\sum_{k=k_0}^n (n-k+1)\rr-3n^5(0.96)^k.
\ee
We are finally ready to choose $k_0$, and do this in a way so that, e.g.,
\be
3n^5(0.96)^k\le n^6(0.96)^k\le \frac{1}{n^2}.
\ee
Elementary algebra shows that (27) surely holds whenever
\be
k\ge 200\ln n:=k_0,
\ee
so that \begin{eqnarray*}\e(X)&\ge& \lr\sum_{k=200\ln n}^n(n-k+1)\rr-\frac{1}{n^2}=\frac{(n-200\ln n)^2}{2}(1-o(1))-\frac{1}{n^2}\\&=&\frac{n^2}{2}(1-o(1)),\end{eqnarray*}
proving the main theorem.
\section{The Polynomial Method of AA, VD, SP, CS, and LY}  Five of the authors of this paper, all REU students in the summer of 2020, are working on the final details of an alternative proof of the Main Theorem.  Their paper will appear elsewhere, but the key idea is to enumerate the number of good permutations on $[k]$, which are polynomial of degree $k-l$.

\section{Acknowledgments}  This research was supported by NSF Grant 1852171.
	
\end{document}